\documentclass[12pt,reqno]{amsart}

\usepackage[latin1]{inputenc}
\usepackage{graphics,color,enumerate}
\usepackage{amssymb,amsmath,amsthm,amscd}
\usepackage{latexsym,verbatim,graphicx,amsfonts}
\usepackage{hyperref}
\usepackage[mathscr]{euscript}

\theoremstyle{plain}
\newtheorem{theorem}{Theorem}[section]

\newtheorem{lemma}[theorem]{Lemma}

\theoremstyle{definition}

\newtheorem{definition}[theorem]{Definition}
\newtheorem{example}[theorem]{Example}
\theoremstyle{remark}

\newcommand{\R}{\mathbb{R}}
\newcommand{\C}{\mathbb{C}}

\newcommand{\T}{\mathbb{T}}
\newcommand{\D}{\mathbb{D}}
\newcommand{\N}{\mathbb{N}}

\DeclareMathOperator{\spn}{Span}
\DeclareMathOperator{\sgn}{\mathrm{sgn}}
\newcommand{\norm}[1]{\left\lVert#1\right\rVert}

\subjclass[2020]{Primary 30C15; Secondary 30H10, 47A16.}

\begin{document}

\title[Equidistribution of zeros]{Equidistribution of zeros of some polynomials related to cyclic functions}
\author[Acuaviva]{Antonio Acuaviva}
\address{Universidad Complutense de Madrid, Facultad de Matem\'aticas, Plaza de Ciencias 3, 28040 Madrid (Madrid), Spain. } \email{ahacua@gmail.com}
\author[Seco]{Daniel Seco}
\address{Universidad Carlos III de Madrid and Instituto de Ciencias Matem\'aticas, Departamento de Matem\'aticas, Avenida de la Universidad 30, 28911 Legan\'es (Madrid), Spain.} \email{dseco@math.uc3m.es}

\date{\today}

\begin{abstract}
In the study of the cyclicity of a function $f$ in reproducing kernel Hilbert spaces an important role is played by sequences of polynomials $\{p_n\}_{n\in \N}$ called \emph{optimal polynomial approximants} (o.p.a.). For many such spaces and when the functions $f$ generating those o.p.a. are polynomials without zeros inside the disk but with some zeros on its boundary, we find that the weakly asympotic distribution of the zeros of $1-p_nf$ is the uniform measure on the unit circle.
\end{abstract}

\maketitle

\section{Introduction}\label{Introduction}

Let $\omega=\{\omega_k\}_{k \in \N}$ be a positive sequence. The \emph{weighted Hardy space} with weight $\omega$ is the space $H^2_\omega$ of holomorphic functions $f$ over the unit disk  $\mathbb{D}$ of the complex plane, given by $f(z) = \sum_{k=0}^{\infty} a_k z^k$ with
\begin{equation}
    \norm{f}^2_\omega := \sum_{k=0}^{\infty} |a_k|^2 \omega_k < \infty.
\end{equation}

A special subclass is formed by the family of weights $\omega_k = (k + 1)^\alpha$, for some $\alpha \in \R$, usually referred to as \emph{Dirichlet-type spaces}. The values $\alpha=-1,0,1$ receive their own names and symbols: the \emph{Bergman space} $A^2$, the \emph{Hardy space} $H^2$ and the \emph{Dirichlet space} $\mathcal{D}$. We refer the reader to the monographs \cite{HKZ, Gar, EFKMR} respectively for the basics about these spaces.

The shift operator $S$ is the operator defined on $H^2_\omega$ by $Sf(z)=zf(z)$ and we say that a (closed) subspace $M \subset H^2_\omega$ is \emph{invariant} if $SM\subset M$. In the theory of invariant subspaces, to which all the above references dedicate a significant effort, a natural question is that of identifying \emph{cyclic functions}, that is, functions that are elements of a space but are contained in no proper invariant subspace. If a function has zeros inside the disk, it is simple to disprove its cyclicity, while if it converges beyond the boundary and has no zeros on the closed disk, it is automatically cyclic. Here, we focus on the \emph{critical} functions, in the sense of having some zeros on the boundary but none inside the disk.

In the last decade, an approach to the study of cyclic functions has been developed \cite{BCLSS} based on the observation that the function $1$ is always cyclic, and therefore a function $f \in  H^2_\omega$ is cyclic if and only if there exists a sequence of polynomials such that \[\|1-p_nf\|_\omega \rightarrow 0, \text{ as } n \rightarrow \infty,\] thus motivating the study of the following family of polynomials. Denote by $\mathcal{P}_n$ the space of polynomials of degree at most $n$.

\begin{definition}\label{opa}
Let $f \in H^2_\omega$, and $n \in \N$. We say that $p_n$ is the \emph{optimal polynomial approximant} (or \emph{o.p.a.}) to $1/f$ of degree less or equal to $n$ if $\|1-p_nf\|_\omega \leq \|1-qf\|_\omega$ for any $q \in \mathcal{P}_n$.
\end{definition}
If $f$ is not identically null, the existence and uniqueness of o.p.a. follow from the fact that $1-p_n f$ is the vector joining $1$ with its orthogonal projection onto the finite-dimensional subspace $\mathcal{P}_n f$. Already in \cite{BCLSS}, it is hinted that the distribution of zeros of o.p.a. or of $\{1-p_nf\}_{n\in \N}$, where $p_n$ are the o.p.a. to $1/f$, may hide relevant information about the cyclicity of the function $f$ and in the present article our goal is to understand this distribution of the zeros of $1-p_nf$ for large values of $n$ and simple functions $f$. In \cite{BenRMI}, it was shown that the zeros of o.p.a. for functions that have no zeros inside the domain $\D$ but with zeros on its boundary, can only accumulate to points of the boundary $\T$ and that indeed every point of the circle is such an accumulation point. Here we will complete that information, showing that if $f \in \mathcal{P}_d$ is critical (in the sense mentioned before), then (a subsequence of) its o.p.a. $\{p_n\}_{n\in \N}$ have the property that the zeros of $1-p_nf$ asymptotically equidistribute over the unit circle (in the weak sense).

Throughout the present text, we will assume a few properties for our weights:
\begin{definition}
We say that $\omega = \{\omega_k\}_{k\in\mathbb{N}}\subset \mathbb{R}^+$ is a \emph{weight} whenever it is a monotone sequence, normalized to have $\omega_0 = 1$, satisfying 

\begin{equation}\label{ratio-condition}
    \lim_{n \to \infty}  \frac{\omega_n}{ \omega_{n - \sqrt{n}}} = 1,
\end{equation}
as well as
\begin{equation}\label{nondegeneracy}
\sum_{k=1}^\infty \frac{1}{\omega_k} = + \infty.
\end{equation}
\end{definition}

It is relevant to notice that the Dirichlet-type spaces do meet our assumptions. The first condition guarantees that ratios of weights converge sufficiently fast to 1. Moreover, one can check that this condition implies the more natural
\begin{equation}\label{limit_to_1}
    \lim_{k\to\infty} \frac{\omega_k}{\omega_{k+1}} = 1.
\end{equation}
In particular, both the shift and its left-inverse are bounded operators in all $H^2_\omega$ spaces and functions in $H^2_\omega$ are naturally associated to the unit disk as their domain of analyticity.
The condition \eqref{nondegeneracy} is shown in \cite{FMS1} to distinguish between the cases in which simple critical functions like $f(z)=1-z$ are cyclic or not (we choose the case in which they are, when cyclicity is a richer phenomenon). We do need some doubling conditions like monotonicity or \eqref{ratio-condition} for technical reasons, in order to avoid spaces with pathological multiplicative behavior. In addition, notice that monotone sequences meeting \eqref{ratio-condition} and \eqref{nondegeneracy} can't grow too fast: indeed, we must have that for each $\varepsilon >0$ there exists some $C_\varepsilon$ such that
\begin{equation}\label{ratio-growth-condition}
 \sup_{k \in [0,n]}    \frac{\omega_n}{\omega_k} \leq C_\varepsilon n^{1+\varepsilon}.
\end{equation}

Denote by $Z(q)$ the zero set of a polynomial $q$ of degree $d \in \N$, and going forward we denote by $\mu_q$ the measure \[\mu_q:= \frac{1}{d} \sum_{z_j \in Z(q)} \delta_{z_j}.\] Finally, denote by $\nu_E$ the uniform measure over $E$. We say that the zeros of the family of polynomials $\{q_n\}_{n\in \N}$ are \emph{asymptotically equidistributed} over $E$ if $\mu_{q_n}$ converges weakly to $\nu_E$.
 Our main result is the following:

\begin{theorem}\label{mainth}
Let $f$ be a critical polynomial with simple zeros, and $p_n$, the $n$-th o.p.a. to $1/f$. For some subsequence $\{n_k\}_{k\in \N} \subset \N$, the zeros of  $\{1-p_{n_k}f\}_{k\in \N}$ are asymptotically equidistributed on $\T$. 
\end{theorem}

Our reasoning will make a strong use of the summary of techniques for (weak) asymptotic equidistribution in \cite{Soundar}. According to E. A. Rakhmanov \cite{Rakh}, the development of the programme to study a cyclic function through its o.p.a. will require understanding the \emph{strong} asymptotics of the zeros, but we consider our contribution an initial step in that direction.
To establish this theorem, we will expand upon the work in \cite{BMS1}, where the boundary behaviour of o.p.a. was studied for $f$, a polynomial of degree $d$ with simple roots. Some results were proved there in full generality and some others, exclusively for the Hardy and Bergman spaces. Here we will need to extend some of those results to general spaces. In particular we will prove results which generalize Theorems 1.7 and 1.8 in \cite{BMS1}. To do so, we need to introduce one more function space. Denote by $\mathcal{A}(\T)$ the Wiener algebra, that is, the space of holomorphic functions over the disk with absolutely summable coefficients, where the norm of a function $g(z) = \sum_{k \in \N} a_k z^k$ is given by \[\|g\|_{\mathcal{A}(\T)} = \sum_{k \in \N} |a_k| < \infty.\] The Wiener algebra is formed by functions with well defined boundary values, and convergence in its norm implies uniform convergence over the closed unit disk. The first result we generalize takes the following form:

\begin{theorem}\label{theorem-1}
Let $f$ be a polynomial with simple zeros such that $Z(f) \cap \mathbb{D} = \emptyset$, and let $p_n$ be the $n$-th o.p.a. to $1/f$ in $H^2_\omega$. Then there exists a constant $C > 0$ such that for all $n \in \mathbb{N}$,
\begin{equation}\label{wiener1}
    \norm{1 - p_n f}_{\mathcal{A}(\T)} \leq C.
\end{equation}
\end{theorem} 

Notice how \eqref{wiener1} couldn't be improved to the left-hand side converging to 0, since values of $1-p_nf$ at the zeros of $f$ on $\T$ can't converge to 0. 

The second result we need to extend deals with pointwise convergence of $1-p_nf$ outside of the zero set of $f$. The answer is contained in the following theorem:

\begin{theorem}\label{theorem-2}
Let $f$ be a polynomial with simple zeros such that $Z(f) \cap \D = \emptyset$, and let $p_n$ be the $n$-th o.p.a. to $1/f$ in $H^2_\omega$. Then
\begin{equation*}
    1 - p_n f \to 0 \text{ as } n \to \infty,
\end{equation*}
uniformly on compact subsets of $\overline{\D} \backslash Z(f)$.
\end{theorem}

In Section \ref{section2} we show how to derive Theorems \ref{theorem-1} and \ref{theorem-2}. These will be used in Section \ref{section3} to establish Theorem \ref{mainth}. These proofs will depend on some technical lemmas that we will state in the relevant place but leave for Section \ref{section4}. We conclude with some remarks on future research on Section \ref{sect5}.

\section{Wiener norm and pointwise convergence}\label{section2}

The spaces we are considering, $H^2_\omega$, are examples of Reproducing Kernel Hilbert Spaces (RKHS) over the disk with an inner product defined by
\begin{equation*}
    \langle f, g \rangle_\omega = \sum_{k=0}^{\infty}a_k\overline{b}_k\omega_k,
\end{equation*}
where $f(z) = \sum_{k=0}^{\infty} a_kz^k$, $g(z) = \sum_{k=0}^{\infty} b_kz^k$. See \cite{FMS1} for the details. RKHS have the special property that norm convergence implies pointwise convergence for points in the common domain of the space (in our case, $\D$).
The reproducing kernel $k(z,b)$ at a point $b \in \D$ is given by
\begin{equation*}
    k(z,b) = \sum_{k=0}^{\infty} \frac{\overline{b}^k z^k}{\omega_k}.
\end{equation*}
Moreover, if we focus on polynomials of degree at most $n$, we can project the reproducing kernel onto $\mathcal{P}_n$, by simply truncating it, to obtain $k_n(z,b)$, with
\begin{equation*}
    k_n(z,b) = \sum_{k=0}^{n} \frac{\overline{b}^k z^k}{\omega_k},
\end{equation*}
which is the reproducing kernel in the subspace $\mathcal{P}_n$ of $H^2_\omega$ (with the inherited norm). We also denote
\begin{equation*}
    S_{n} := \sum_{k=0}^{n} \frac{1}{\omega_k}.
\end{equation*}
Note that $S_{n}$ could be interpreted as the value of $k_n(z,z)$ if this was defined for any $z \in \mathbb{T}$, and it gives us an upper bound for the values that the reproducing kernels $k_n(z,b)$ can attain on the unit disk. From now on, whenever we write $\hat{g}(k)$ we mean the Taylor coefficient of order $k$ of the function $g$, and when we write $v^t$ we mean the transpose of $v$.

A key result about o.p.a. in our context was proved in \cite{BMS1}, Corollary 1.2, which provides a closed formula for all the coefficients of all degrees for $1-p_nf$:
\begin{lemma}\label{projections}
Let $f$ be a monic polynomial of degree $d$ with simple zeros $z_1, \dots, z_{d}$ that lie in $\C \backslash \{0\}$, $p_n$ the $n$-th o.p.a. to $1/f$ in $H^2_\omega$, $d_{k,n} = (\widehat{1-p_nf})(k)$, $v_0 = (1, \dots, 1) \in \C^d$, and $E = E_{Z,n} := (e_{l,m})_{l,m=1}^d$ be the matrix whose coefficients are given by $e_{l,m} = k_{n+d}(z_l, z_m)$. Then there exists a unique vector $A_n = (A_{1,n}, \dots, A_{d,n})$ such that for $k = 0, \dots, n+d$ we have
\begin{equation}
    d_{k,n} = \frac{1}{\omega_k}\sum_{i=1}^d A_{i,n}\overline{z}_i^k.
\end{equation}
Moreover,  \[A_n^t = E^{-1} \cdot v_0^t.\]
\end{lemma}

Furthermore, moving forward, we will assume that $f$ has simple zeros $z_1, \dots, z_{d_1} \in \mathbb{T}$, $d_1 \geq 1$ and $z_{d_1 + 1}, \dots, z_{d} \in \mathbb{C}\backslash \overline{\mathbb{D}}$ and that they are ordered, that is, $|z_i| \leq |z_{i+1}|$.

In order to prove Theorems \ref{theorem-1} and \ref{theorem-2}, we need bounds for $A_{i,n}$, to then apply the previous lemma. We need some preliminary results.

First we need to control the size of truncated reproducing kernels, as these provide the elements of the matrix $E$. Whenever $|\overline{z}_j z_i| > 1$, this is contained in the following lemma:

\begin{lemma}\label{lemma_k}
Let $z \in \C$ with $|z| > 1$, and $N \in \N$. Then \[\sum_{k=0}^{N} \frac{z^k}{\omega_k} = C(N,z)\frac{z^{N+1}}{\omega_N}\] where \[\lim_{N\rightarrow\infty} C(N,z) = \frac{1}{z-1}.\]
\end{lemma}

\begin{proof}
Define
\begin{equation*}
    C(N,z) = \frac{\sum_{k=0}^N \frac{z^k}{\omega_k}}{\frac{z^{N+1}}{\omega_N}}  = \sum_{k=0}^N \frac{\omega_Nz^{-N-1+k}}{\omega_k}.
\end{equation*}
Under the change of variables $x= \frac{1}{z}$, we can rewrite this as
\begin{equation*}
    C(N,x)  = \sum_{k=0}^N \frac{\omega_N x^{N+1-k}}{\omega_k} = x \sum_{k=0}^N \frac{\omega_N x^{k}}{\omega_{N-k}},
\end{equation*}
where $|x| < 1$.
Now let $\epsilon > 0$ be fixed. We split the sum in two parts
\begin{equation*}
        C(N,x) = x\sum_{k=0}^{\sqrt{N}}\frac{\omega_N}{\omega_{N-k}}x^k + x^{\sqrt{N} + 1}\sum_{k=\sqrt{N}+1}^{N}\frac{\omega_N}{\omega_{N-k}}x^{k-\sqrt{N}}.
\end{equation*}
Therefore, we can compute the difference
\begin{equation*}
\begin{split}
        \left|C(N,x) - \frac{x}{1-x}\right| \leq \left|C(N,x) - x\sum_{k=0}^{\sqrt{N}}x^k\right| + \left|x\sum_{k=0}^{\sqrt{N}}x^k - \frac{x}{1-x}\right|.
\end{split}
\end{equation*}
The second term of the right-hand side can be made smaller than an arbitrary $\epsilon$ by taking a large value of $N$. Meanwhile, the first term can easily be bounded above by
\begin{equation*}
 |x|\left|\sum_{k=0}^{\sqrt{N}}\left(\frac{\omega_N}{\omega_{N-k}}-1\right)x^k\right| + |x|^{\sqrt{N}+1}\hspace{2pt}\left|\sum_{k=\sqrt{N}+1}^{N}\frac{\omega_N}{\omega_{N-k}}x^{k-\sqrt{N}}\right|.
\end{equation*}
From condition (\ref{ratio-condition}) and the fact that the sequence of weights is monotone, it follows that choosing $N$ large enough, we can derive
\begin{equation*}
    \sup_{k \in \{N - \sqrt{N}, \dots, N\}} \left|\frac{\omega_N}{\omega_{N-k}} - 1 \right| \leq \epsilon(1-|x|).
\end{equation*}
Then we can bound
\begin{equation*}
    |x|\left|\sum_{k=0}^{\sqrt{N}}\left(\frac{\omega_N}{\omega_{N-k}}-1\right)x^k\right| \leq \epsilon.
\end{equation*}

We are finally left with the term on $k \geq \sqrt{N}$, which tends to zero as it is exponentially suppressed in $|x|^{\sqrt{N}+1}$ and the ratio growth is bounded by \eqref{ratio-growth-condition}, therefore, we can make this term smaller than $\epsilon$ for sufficiently large $N$. Reversing the change $x = \frac{1}{z}$ the result follows.
\end{proof}

Notice that here the Szeg\"o kernel (the reproducing kernel for the classical Hardy space $H^2$) plays a universal role among weighted Hardy spaces. We are also interested in the size of $k_n(z_i,z_j)$ when $z_i \neq z_j$, $|\overline{z}_j z_i| \leq 1$. Its size relative to $S_n$ is determined in the next lemma.

\begin{lemma}\label{k_small}
Let $z_1, z_2 \in \overline{\D}$, $z_1 \neq z_2$. 
Then \[\lim_{n \to \infty} \frac{k_n(z_1,z_2)}{S_n} = 0.\]
\end{lemma}

\begin{proof}
By definition $\left|k_n(z_1,z_2) \right| = \left|  \sum_{k=0}^n \frac{\overline{z}_2^k z_1^k}{\omega_k}  \right|$. Applying summation by parts and calling $A_l = \sum_{k=0}^l \overline{z}_2^k z_1^k = \frac{1-\overline{z}_2^{l+1}z_1^{l+1}}{1-\overline{z}_2 z_1}$, it follows that
\begin{align*}
\left|  \sum_{k=0}^n \frac{\overline{z}_2^k z_1^k}{\omega_k}  \right| &= \left|\frac{A_n}{\omega_n} - \sum_{k=0}^{n-1}A_k\left(\frac{1}{\omega_{k+1}} - \frac{1}{\omega_k}\right)\right| \\
& \leq \frac{2}{|1-\overline{z}_2 z_1|}\left(\frac{1}{\omega_n} + \sum_{k=0}^{n-1}\left|\frac{1}{\omega_{k+1}} - \frac{1}{\omega_k}\right|\right).
\end{align*}
Using that the sequence of weights is monotone, we can see
\begin{equation*}
\left|k_n(z_1,z_2) \right| \leq  \frac{2}{|1-\overline{z}_2 z_1|}\left(\frac{2}{\omega_n} + 1 \right).
\end{equation*}
From the limit condition (\ref{limit_to_1}) and the divergence of $S_n$ as $n \to \infty$, the fraction on the right-hand side  goes to zero.
\end{proof}

We can now state an estimate the size of $A_{i,n}$. Its proof will require the use of the intermediate Lemma \ref{lemma-3.2}. 

\begin{lemma}\label{main-lemma}
As $n \to \infty$, the coefficients $A_{i,n}$ meet the following rates of decay:
\begin{align*}
    A_{i,n} \in 
     \begin{cases}
       O\left(\frac{1}{S_{n+d}}\right) \text{ for } 1 \leq i \leq d_1 \\
       O\left(\frac{1}{S_{n+d}|z_i|^{n+d+1}}\right) \text{ for } d_1 < i \leq d.\\
     \end{cases}
\end{align*}
Consequently, for each $1 \leq i \leq d_1$, there exists a constant $C_i$, independent of $n$, such that
\begin{equation*}
    \sum_{k=0}^{n+d}\left|A_{i,n}\frac{\overline{z_i}^k}{w_k} \right| \leq C_i,
\end{equation*}
while for $d_1 < i \leq d$, we have
\begin{equation*}
    \sum_{k=0}^{n+d}\left|A_{i,n}\frac{\overline{z_i}^k}{w_k} \right| \to 0
\end{equation*}
as $n \to \infty$.
\end{lemma}

Before we can show Lemma \ref{main-lemma} we need to estimate determinants, in particular of the matrix $E$ in the statement of the Lemma \ref{projections}, \emph{from below}. This is reasonable since $E$ is a positive definite matrix but estimating determinants from below is necessarily painful. This use of the determinants to bound $A_{i,n}$ is in the spirit of what's done in Lemma 3.3 and Lemma 5.4 from \cite{BMS1} for the $H^2$ and $A^2$ spaces respectively. Our lower estimate for the determinant is also based on a similar technique in Lemma 3.2 from the same article. 

\begin{lemma}\label{lemma-3.2}
Let $1 \leq d_1 \leq d$ be integers, and $\{z_i\}_{i=1}^d \subset \mathbb{C}$ be distinct points with $|z_i| = 1$ for $1 \leq i \leq d_1$ and $|z_i| > 1$ for $d_1 < i \leq d$. If $E := (e_{l,m})^d_{l,m=1}$ with $e_{l,m} = k_{n+d}(z_l, z_m)$, then there exists a constant $\delta > 0$, independent of n, such that for every n,
\begin{equation}
    \det(E) \geq \delta \frac{S_{n+d}^{d_1}}{w_{n+d}^{d - d_1}} \prod_{l = 1}^d |z_l|^{2(n+d+1)}.
\end{equation}
\end{lemma}

\begin{proof}
In what follows, we use the notation $\mathscr{S}$ to denote the set of all permutations of the indices $\{1,\dots, d\}$, $\sgn{(\sigma)}$ to denote the parity of a particular permutation and $id$ to refer to the identity permutation. By the definition of determinant,
\begin{equation*}
    \det{(E)} = \sum_{\sigma \in \mathscr{S}}\left[\sgn{(\sigma)}\prod_{l=1}^d e_{l, \sigma(l)}\right] = \sum_{\sigma \in \mathscr{S}}\left[\sgn{(\sigma)}\prod_{l=1}^d \left(\sum_{k=0}^{n+d} \frac{\overline{z}_{\sigma(l)}^k z_l^k }{\omega_k}\right)\right].
\end{equation*}

We can decompose this sum depending on the number of indices that a given permutation fixes. Recall that when $i = 1, \dots, d_1$ then $|z_i| = 1$ while $|z_i| > 1$ otherwise. Let $\mathscr{A}$ be the set of permutations such that $\sigma(i) = i$ for every $1 \leq i \leq d_1$ and for each $0 \leq j \leq d_1$ let $\mathscr{B}_j$ be the set of permutations that fix exactly $j$ of the indices in the set ${1, \dots, d_1}$. Then,
\begin{align*}
    \det{(E)} = \prod_{l=1}^d \left(\sum_{k=0}^{n+d} \frac{|z_l|^{2k}}{\omega_k}\right) + \sum_{\sigma \in \mathscr{A}\backslash \{id\}} \sgn{(\sigma)} \prod_{l=1}^d \left(\sum_{k=0}^{n+d} \frac{\overline{z}_{\sigma(l)}^k z_l^k}{\omega_k}\right) \\
    + \sum_{j=0}^{d_1 - 1} \sum_{\sigma \in \mathscr{B}_j} \sgn{(\sigma)}\prod_{l=1}^d \left(\sum_{k=0}^{n+d} \frac{\overline{z}_{\sigma(l)}^k z_l^k}{\omega_k}\right).
\end{align*}

Now, we can first study the sums inside the products. When $l = \sigma(l) \leq d_1$, then 
\begin{equation*}
    \sum_{k=0}^{n+d} \frac{|z_l|^{2k}}{\omega_k} = S_{n+d}.
\end{equation*}
Meanwhile, when $l \not = \sigma(l)$ but both are smaller than or equal to $d_1$, we can invoke Lemma \ref{k_small}, yielding
\begin{equation*}\sum_{k=0}^{n+d} \frac{\overline{z}_{\sigma(l)}^k z_l^k}{\omega_k} \in o(S_{n+d})
\end{equation*}
as $n$ grows to $\infty$. 
Finally, if $l$ or $\sigma(l)$ is bigger than $d_1$, then $|\overline{z}_{\sigma(l)} z_l| > 1$, therefore, using Lemma \ref{lemma_k}, we get 
\begin{equation*}
    \sum_{k=0}^{n+d} \frac{\overline{z}_{\sigma(l)}^k z_l^k}{\omega_k} = C(l, \sigma, n) \frac{\overline{z}_{\sigma(l)}^{n+d+1} z^{n+d+1}_l }{\omega_{n+d}},
\end{equation*}
where
\begin{equation*}
    C(l,\sigma,n) \to \frac{1}{\overline{z}_{\sigma(l)} z_{l} - 1}, \text{ as }n \to \infty.
\end{equation*}

Therefore, the first summand in the expression for the determinant is
\begin{equation}
    \frac{S_{n+d}^{d_1}}{\omega_{n+d}^{d-d_1}} \left(\prod_{l=d_1 + 1}^d |z_l|^{2(n+d+1)}C(l,id, n) \right).
\end{equation}
We can also compute the second summand corresponding to the permutations in $\mathscr{A}$,
\begin{equation}\label{sum-A}
    \sum_{\sigma \in \mathscr{A}\backslash \{id\}} \sgn{(\sigma)} \frac{S_{n+d}^{d_1}}{\omega_{n+d}^{d-d_1}}\left(\prod_{l=d_1 + 1}^d (\overline{z}_{\sigma(l)} z_l)^{n+d+1}C(l,\sigma, n)  \right).
\end{equation}

Finally, the summand corresponding to the permutations in $\mathscr{B}_j$ consist of sums similar to those in (\ref{sum-A}), except involving powers of $S_{n+d}^j$ and products over some subsets of indices $l$.

Notice that for $\sigma \in \mathscr{A}$, since $\sigma$ is bijective from $\{d_1 + 1, \dots, d\}$ to itself, we have
\begin{equation*}
    \prod_{l=d_1 + 1}^d (\overline{z}_{\sigma(l)} z_l)^{n+d+1} = \prod_{l = d_1 + 1}^d |z_l|^{2(n+d+1)} = \prod_{l = 1}^d |z_l|^{2(n+d+1)}.
\end{equation*}
Therefore, we see that the determinant of E is the product of two factors, first one being
\begin{equation*}
\frac{S_{n+d}^{d_1}}{\omega_{n+d}^{d - d_1}} \prod_{l = 1}^d |z_l|^{2(n+d+1)}\end{equation*} and the second one being \begin{equation}\label{eqn1001}
     \prod_{l = d_1 + 1}^d C(l, id, n) + \sum_{\sigma \in \mathscr{A}\backslash \{id\}} \sgn{(\sigma)} \prod_{l = d_1 + 1}^d C(l, \sigma, n) + r(n),
\end{equation}
where $r(n) \to 0$ as $n \to \infty$. Note that $E$ is a Gram matrix, and $\det(E) > 0$ for all $n$. On the other hand as $n \rightarrow \infty$, $r(n) \to 0$, and thus we get that the limit of \eqref{eqn1001} coincides with the determinant of the matrix $B := (b_{l,m})_{l,m = d_1 + 1}^d$, defined by $b_{l,m}= \frac{1}{\overline{z}_m z_l - 1}$ and, as was shown in \cite{BMS1}, this is positive definite so that $\det(B) > 0$. Therefore as $\det(E) > 0$, the second factor is strictly positive for all $n$ and converges to $\det(B) > 0 $, so it is bounded below by a constant $\delta > 0$. Hence we obtained the desired result. \end{proof}

With this control on the determinant, we are finally ready to give bounds for the size of $A_{i,n}$. We will combine Lemmas \ref{lemma_k}, \ref{k_small},  and \ref{lemma-3.2} and the proof of a similar result in \cite{BMS1}.

\begin{proof}[Proof of Lemma \ref{main-lemma}.] The coefficients $A_{i,n}$ are the solutions to the linear system $E \cdot A_n^t = v_0^t$, where $v_0 := (1, \dots, 1) \in \mathbb{C}^d$. Therefore, by Cramer's rule we have that $A_{i,n} = \frac{\det(E^{(i)})}{\det(E)}$ where $E^{(i)}$ is obtained replacing the i-th column of E by $v_0^t$.

Now if $1 \leq 1 \leq d_1$, arguing as in Lemma \ref{lemma-3.2}, in all the sums $\sigma(l) \not = i$, the highest power of $S_{n+d}$ that can appear in any term of the expression of $\det(E^{(i)})$ is $S_{n+d}^{d_1 - 1}$, multiplied by a product that is bounded above by a constant multiple of $\frac{1}{\omega_{n+d}^{d - d_1}}\prod_{l=1}^d |z_l|^{2(n+d+1)}$. There thus exists a positive constant $C_1$ such that 
\begin{equation*}
    |\det(E^{(i)})| \leq C_1 \frac{S_{n+d}^{d_1-1}}{w_{n+d}^{d - d_1}} \prod_{l = 1}^d |z_l|^{2(n+d+1)}.
\end{equation*}
Applying Lemma \ref{lemma-3.2} gives that for $1 \leq i \leq d_1$, we have that $A_{i,n} \in O\left(\frac{1}{S_{n+d}}\right)$.

In the case $ d_1 < i \leq d$ arguing as in Lemma \ref{lemma-3.2},
\begin{eqnarray}
\label{eq-id2} \det{(E^{(i)})} = \frac{S_{n+d}^{d_1}}{\omega_{n+d}^{d-d_1+1}}\prod_{\substack{l = d_1 + 1 \\ l \not = i}}^d |z_l|^{2(n+d+1)} C(l,id,n) \\ \label{eq-A2}
+\sum_{\sigma \in \mathscr{A} \backslash \{id\}} \sgn{(\sigma)} \frac{S_{n+d}^{d_1}}{\omega_{n+d}^{d-d_1+1}} \prod_{\substack{l = d_1 + 1 \\ \sigma(l) \not = i}}^d (\overline{z}_{\sigma(l)}z_l)^{n+d+1} C(l,\sigma,n) \\ \label{eq-R2}
+ R(n)
\end{eqnarray}
where $R(n)$ denotes the remainder terms. Now, (\ref{eq-id2})  is missing a term of order $\frac{1}{w_{n+d}}|z_i|^{2(n+d+1)}$, while in (\ref{eq-A2}) each product is missing a term of order $\frac{1}{w_{n+d}}|\overline{z}_i z_{i^*}|^{n+d+1}$ where $i^* = \sigma^{-1}(i) >d_1$. Finally, in (\ref{eq-R2}), the highest power of $S_{n+d}$ that appears is $S_{n+d}^{d_1 - 1}$, and each product is missing at least one term of order $|z_i|^{n+d+1}$.  Therefore after dividing by $\det{(E)}$ and using Lemma \ref{lemma-3.2} we can conclude that $A_{i,n}$ has an order of decay at most
\begin{equation*}
    O\left(\frac{w_{n+d}}{|z_i|^{2(n+d+1)}} + \frac{\omega_{n+d}}{|z_i|^{n+d+1} |z_{d_1 + 1}|^{n+d+1}} + \frac{1}{S_{n+d}|z_i|^{n+d+1}}\right),
\end{equation*}
where we have used that $|z_{d_1+1}| \leq |z_{d_1 + j}|$, $j = 1, \dots, d-d_1$ combined with $\sigma^{-1}(i) > d_1$.
Therefore, given that the growth of $w_{n+d}$ is not exponential, it follows that
$A_{i,n} \in O\left( \frac{1}{S_{n+d}|z_i|^{n+d+1}}\right)$ as desired.

For the second part of the lemma for $1 \leq i \leq d_1$, there is a constant $C_i$ such that
\begin{equation*}
    \sum_{k=0}^{n+d} \left|A_{i,n} \frac{\overline{z}_i^k}{\omega_k}\right| \leq \frac{C_i}{S_{n+d}} \sum_{k=0}^{n+d} \frac{|\overline{z}_i|^k}{w_k} = C_i,
\end{equation*}
and for $d_1 < i \leq d$ we have, using Lemma \ref{lemma_k},
\begin{equation*}
    \sum_{k=0}^{n+d} \left|A_{i,n} \frac{\overline{z}_i^k}{\omega_k}\right| = |A_{i,n}| C(n+d,|\overline{z}_i|) \frac{|\overline{z}_i|^{n+d+1}}{\omega_{n+d}}
\end{equation*}
where $C(n+d,|\overline{z}_i|) \to \frac{1}{|\overline{z}_i| - 1}$. Hence, $\frac{1}{S_{n+d} \omega_{n+d}} \to 0$ as $n\rightarrow \infty$, implying
\begin{equation*}
        |A_{i,n}| C(n+d,|\overline{z}_i|) \frac{|\overline{z}_i|^{n+d+1}}{\omega_{n+d}} \to 0.
\end{equation*}
\end{proof}

Now we are ready to show the validity of Theorem \ref{theorem-1}.
\begin{proof}[Proof of Theorem \ref{theorem-1}.]
Recall from Lemma \ref{projections} that $(1-p_nf)(z) = \sum_{k=0}^{n+d} d_{k,n} z^k$, where $d_{k,n} = \frac{1}{\omega_k}\sum_{i=1}^d A_{i,n} \overline{z}_i^k$. Therefore, we can estimate the Wiener norm as:
\begin{equation*}
    \norm{1 - p_n f}_{\mathcal{A}(\T)} \leq \sum_{i=1}^d \sum_{k=0}^{n+d} \left|A_{i,n} \frac{\overline{z}_i^k}{\omega_k}\right|.
\end{equation*}
Now invoking Lemma \ref{main-lemma}, we can conclude that
\begin{equation*}
    \norm{1 - p_n f}_{\mathcal{A}(\T)} \leq \sum_{i=1}^{d_1} C_i + o(1) \leq C < \infty
\end{equation*}
for some positive constant $C$, as desired.
\end{proof}

The proof of Theorem \ref{theorem-2} can also be performed at this point.
\begin{proof}[Proof of Theorem \ref{theorem-2}.]
Let $K \subset \overline{\mathbb{D}} \backslash \{z_1, \dots, z_d\}$ be a compact set. Then, by Lemma $\ref{projections}$, for each $z \in K$ we have
\begin{align*}
    (1-p_nf)(z) &= \sum_{k=0}^{n+d} \left(\sum_{i=1}^d A_{i,n} \frac{\overline{z}_i^k}{\omega_k}\right) z^k = \sum_{i=1}^d A_{i,n} \left( \sum_{k=0}^{n+d} \frac{\overline{z}_i^k z^k}{\omega_k}\right) \\
    &= \sum_{i=1}^d A_{i,n} k_{n+d}(z, z_i).
\end{align*}
Now, for $1 \leq i \leq d_1$, we have $A_{i,n} \in O\left(\frac{1}{S_{n+d}}\right)$ by Lemma \ref{main-lemma}, while by Lemma \ref{k_small}, $k_{n+d}(z, z_i) \in o(S_{n+d})$ uniformly on $z \in K$ as we are avoiding a neighbourhood of $z_i$, $1 \leq i \leq d_1$. Therefore these terms go uniformly to zero on $K$.

For the case $d_1 < i \leq d$, one can see that $k_{n+d}(z, z_i) \in O\left(\frac{\overline{z}_i^{n+d+1}}{w_{n+d}}\right)$ uniformly on $K$.
Again, by Lemma $\ref{main-lemma}$, $A_{i,n} \in O\left(\frac{1}{S_{n+d}|z_i|^{n+d+1}}\right)$. As a result, these terms also go uniformly to zero on $K$.
\end{proof}

\section{Distribution of zeros}\label{section3}

We are going to show the proof of Theorem \ref{mainth}. We will exploit classical results from approximation theory, for which we were inspired by the nice summary of techniques in \cite{Soundar}. For the rest of the article, for a polynomial $P$, $\mu_P$ will denote the measure formed as the average of  the delta measures at the zeros of $P$, and $\nu_E$ will denote the uniform distribution measure over the set $E$.
We base our approach on a classical result of \"Erdos and Tur\'an, claiming that, given a monic polynomial with small size on the unit circle and not too small of a constant coefficient, then its zeros cluster uniformly around the unit circle. There are multiples ways of quantifying these conditions for a monic polynomial. For our case, given a polynomial $P$ we will study
\begin{equation*}
    H(P) = \max_{|z| = 1} \frac{|P(z)|}{\sqrt{|P(0)|}}.
\end{equation*}
With respect to this $H$, \"Erdos-Tur\'an's result is stated as follows:

\begin{theorem}\label{root-distribution-theorem}
Let $\{P_n\}_{n \in \N}$ be a family of monic polynomials, such that $H(P_n) = e^{o(n)}$. Then   
\begin{equation*}
  \lim_{n \to \infty} \mu_{P_{n}} = \nu_{\T},
\end{equation*}
where the convergence is in the weak sense.
\end{theorem}

See \cite{Soundar}. It will turn out that we have plenty of room to establish the applicability of this Theorem to our context, but we prefer to get the best estimates we may obtain, in order to promote future further research.
To apply the result to our case, we just need to re-normalize our polynomials into monic polynomials. We focus on $P_n = 1 - p_nf$, which is a polynomial of degree, at most, $n + d$, and study the polynomial $\frac{P_n}{d_{n + d, n}}$ which is, obviously, monic and has the same roots as $P_n$. We have that
\begin{equation*}
    H\left(\frac{P_n}{d_{n+d,n}}\right) = \max_{|z| = 1} \frac{|P_n(z)|}{\sqrt{|d_{0,n}d_{n+d,n}|}} \leq \frac{C}{\sqrt{|d_{0,n}d_{n+d,n}|}}
\end{equation*}
where we used Theorem \ref{theorem-1} to bound the value of $P_n$ on the unit disk. Therefore, in order to bound the values for $H$, we need to find lower bounds for the values of $d_{0,n}$ and $d_{n+d,n}$. The first one was already done in \cite{FMS1}, Theorem 4.1 showing that \[\|1-p_n f\|^2 \approx \frac{1}{S_{n+d}},\] where $A \approx B$ means that there exists a constant $C$ such that $ C^{-1} B \leq A \leq CB$. This is enough for determining $d_{0,n}$ up to a constant since $p_nf$ is the orthogonal projection of $1$ onto $\mathcal{P}_n f$ and thus \[\|1-p_nf\|^2 = \left\langle 1-p_nf, 1-p_nf \right\rangle =   \left\langle 1-p_nf, 1 \right\rangle = d_{0,n}.\]

Only $d_{n+d,n}$ remains to be correctly estimated. Let us see how this works with an example that will illustrate both the difficulties and the solution for this problem. After this we will provide a general proof.

\begin{example} \label{example-degree-two}
Let $f$ be a degree two monic polynomial with two distinct roots on the unit circle. Without loss of generality, and under a rotation if necessary, we may assume that $z_1 = e^{i\theta} = \overline{z}_2$ are these two roots and $0 < \theta \leq \pi/2$.

Under these conditions, the matrix $E$ as defined in Lemma \ref{projections} takes the form
\begin{equation*}
    E = \begin{pmatrix}
    S_{n+2} & k_{n+2}(e^{i\theta}, e^{-i\theta}) \\ k_{n+2}(e^{-i\theta}, e^{i\theta}) & S_{n+2} 
    \end{pmatrix}.
\end{equation*}
Therefore, it follows that
\begin{equation*}
    E^{-1} = \frac{1}{\det{(E)}}\begin{pmatrix}
    S_{n+2} & -k_{n+2}(e^{i\theta}, e^{-i\theta}) \\ -k_{n+2}(e^{-i\theta}, e^{i\theta}) & S_{n+2} 
    \end{pmatrix}.
\end{equation*}
So that
\begin{equation*}
    A_{1,n} = \frac{1}{\det{(E)}} \left(S_{n+2} -k_{n+2}(e^{i\theta}, e^{-i\theta}) \right)
\end{equation*}
\begin{equation*}
    A_{2,n} = \frac{1}{\det{(E)}} \left(S_{n+2} -k_{n+2}(e^{-i\theta}, e^{i\theta}) \right).
\end{equation*}
By Lemma \ref{projections}, we have that $d_{n+2, n}$ is equal to
\begin{equation*}
\frac{2}{\omega_{n+2} \det{(E)}} \left[S_{n+2}\cos{(\theta(n+2))} - \Re{\left(k_{n+2}(e^{i\theta}, e^{-i\theta})e^{-i\theta(n+2)} \right)}\right].
\end{equation*}
Now, for any $0 < \theta \leq \pi/2$, there exists an increasing sequence of natural numbers $\{n_k\}_{k=0}^\infty$ and a positive constant $\delta$ with $|\cos{(\theta(n_k + 2))}| > \delta$ for all $n_k$. For this sequence it follows that
\begin{equation*}
    |d_{n_k + 2, n_k}| \geq \frac{2 \delta S_{n_k + 2}(1 + o(1))}{\omega_{n_k+2} S_{n_k+2}^2(1 + o(1)) } \geq \frac{C}{\omega_{n_k+2} S_{n_k+2}}
\end{equation*}
for some constant $C > 0$, thus finding a lower bound for a subsequence of $d_{n+2, n}$.
Therefore it follows that,
\begin{equation*}
    H\left(\frac{P_{n_k}}{d_{n_k+2, n_k}}\right) = O\left(S_{n_k + 2} \sqrt{\omega_{n_k + 2}}\right) \implies  H\left(\frac{P_{n_k}}{d_{n_k+2, n_k}}\right) = e^{o(n_k)},
\end{equation*}
so we can apply Theorem \ref{root-distribution-theorem} to this subsequence. 

Note that, unless we delve into more details, we can't avoid working with subsequences as we can't control $|\cos(\theta(n+2))|$ under general conditions. To make the situation manageable, we will need Lemma \ref{sum-cosines-lemma} below.
\end{example}

Before trying to compute a lower bound for $d_{n+d,n}$ that is valid in general, we will need some technical work. We just state here some technical lemmas and leave their proofs for the next section.

\begin{lemma}\label{d_n-lemma}
Let $f$ be a monic polynomial of degree $d$ with simple zeros $Z(f) \cap \D = \emptyset \neq Z(f) \cap \T = \{z_i\}_{i=1}^{d_1}$ and $\{z_i\}_{i=d_1+1}^{d}= Z(f) \backslash \overline{\D}$. Let $p_n$ be the $n$-th o.p.a. to $1/f$. Define $v := (v_m)_{m=1}^{d-d_1}$, $v_m = 1/\overline{z}_{m+d_1}$, $B := (b_{l,m})_{l,m=1}^{d-d_1}$, $b_{l,m} = \frac{1}{\overline{z}_{m+d_1} z_{l+d_1}-1}$ and $s := (s_l)_{l=1}^{d-d_1}$, $s_l = \sum_{j=1}^{d_1} \frac{\overline{z}_j^{n+d+1}}{\overline{z}_j z_{l+d_1} - 1}$. Then there exists a positive constant $C$ such that
\begin{align*}
    d_{n+d, n} = \widehat{1 - p_{n}f}(n+d) = \frac{C}{w_{n+d}S_{n+d}}\left(G_n + o(1)\right)
\end{align*}
where
\begin{align*}
    G_{n} = \begin{vmatrix}\sum_{j=1}^{d_1} \overline{z}_j^{n+d}& v\\s^t&B\end{vmatrix}.
\end{align*}
\end{lemma}

The next lemma involves the value of the determinant $G_n$ in the special case where we have only one zero in the unit disk, which we will take to be $z_1 = 1$. 

\begin{lemma}\label{g_one_zero}
Let $z_2, \dots, z_d \in \mathbb{C}\backslash \overline{\mathbb{D}}$, be different with $d \geq 2$, then
\begin{align*}
    G = \begin{vmatrix}1& v\\s^t&B\end{vmatrix} \not = 0,
\end{align*}
where $v := (v_m)_{m=2}^{d}$, $v_m = 1/\overline{z}_{m}$, $B := (b_{l,m})_{l,m=2}^{d}$, $b_{l,m} = \frac{1}{\overline{z}_{m} z_{l}-1}$ and $s := (s_l)_{l=2}^{d}$, $s_l = \frac{1}{z_{l} - 1}$.
\end{lemma}

The third of the auxiliary lemmas, inspired by Example \ref{example-degree-two}, deals with the need to control the size of trigonometric functions. Particularly, it is obvious when the quotients between angles are rational:

\begin{lemma}\label{sum-cosines-lemma}
Let $\theta_1, \dots, \theta_n \in [-\pi, \pi)$. Then, there exist $\{n_k\}_{k=0}^\infty \subset \mathbb{N}$ such that
\begin{equation*}
    \lim_{k \to \infty} n_k \theta_k \equiv 0 \hspace{3pt} \mod{2\pi}.
\end{equation*}
\end{lemma}

Our last lemma, for now, is about linear combinations of powers of unimodular complex numbers.

\begin{lemma}\label{sum_exponentials_lemma}
    Let $\theta_1, \dots, \theta_n \in [0, 2\pi)$ be different, $C_1, \dots, C_n$ be any arbitrary complex numbers and $N \in \mathbb{N}$.
    Then
    \begin{equation*}
        \sum_{k=1}^n C_k e^{im\theta_k} = 0, \forall m \geq N, m \in \mathbb{N}
    \end{equation*}
    if and only if $C_1 = C_2 = \dots = C_n = 0$.
\end{lemma}

Using these lemmas, we can finally prove Theorem \ref{mainth}.

\begin{proof}[Proof of Theorem \ref{mainth}.]
It suffices to show that there exists an increasing sequence of naturals numbers $\{n_k\}_{k=0}^\infty$ such that $|d_{n_k + d, n_k}| \geq \frac{\delta}{w_{n+d}S_{n+d}}$ for some $\delta > 0$. Under that condition
\begin{equation*}
    H\left(\frac{P_{n_k}}{d_{n_k+d, n_k}}\right) = O\left(S_{n_k + d}\sqrt{\omega_{n_k + d}}\right) \implies  H\left(\frac{P_{n_k}}{d_{n_k+2, n_k}}\right) = e^{o(n_k)},
\end{equation*}
and the Theorem follows from Theorem \ref{root-distribution-theorem}.
Now, using Lemma \ref{d_n-lemma} it is enough to show that there exists a subsequence $\{n_k\}_{k=0}^\infty$ and $\delta > 0$ such that $|G_{n_k}| > \delta$. 

For that, note that the previous determinant can be simply rewritten as
\begin{equation*}
    G_n = \sum_{j=1}^{d_1}\overline{z}_j^{n+d+1}\begin{vmatrix} 1/\overline{z}_j & v\\s^t_j &B\end{vmatrix}, 
\end{equation*}
where $(s_{j,l})_{l=1}^{d-d_1} = \frac{1}{\overline{z}_j z_{l+d_1} - 1}$.
We distinguish between two cases: when we have zeros outside the unit circle and when we don't have such zeros.

In the case that there are no zeros outside the unit circle, we simply get that $G_n = \sum_{j=1}^{d_1} \overline{z}_j^{n+d} = \sum_{j=1}^{d_1} e^{-i\theta_j(n+d)}$. Applying Lemma \ref{sum-cosines-lemma}, it follows that we can find $\{n_k\}_{k=0}^\infty$ such that the sum can approximate  $\sum_{j=1}^{d_1} e^{-i0} = d_1$.

When we have zeros outside the unit circle, rotating the plane if necessary, we can assume that $z_1 = 1$, so that we can rewrite
\begin{equation*}
    G_n = G + \sum_{j=2}^{d_1}\overline{z}_j^{n+d+1}\begin{vmatrix} 1/\overline{z}_j & v\\s^t_j &B\end{vmatrix},
\end{equation*}
where $G$ is as described in Lemma \ref{g_one_zero}. Therefore, we can apply Lemma \ref{sum_exponentials_lemma} to conclude that there exists $N$ such that $G_N \not = 0$ and use Lemma \ref{sum-cosines-lemma} to find a sequence $\{n_k\}_{k=0}^\infty$ such that $G_{n_{k}} \to G_N \not = 0$, proving the theorem.
\end{proof}

\section{Proof of technical lemmas}\label{section4}

Now we will concentrate on the task of establishing Lemma \ref{d_n-lemma}, for which we will use all the notation in the statement as well as the notation from Lemma  \ref{projections}. This will take the form of one more auxiliary result. 

We also denote by $B_{(j)}^{(i)}$ the matrix created by removing the i-th column of the B matrix and adding the vector $s_j := \left(s_{j,l}\right)_{l=1}^{d-d_1}$, $s_{j,l} = \frac{\overline{z}_j^{n+d+1}}{\overline{z_j} z_{d_1 + l} - 1}$ as its first column. With this notation, we have the following:

\begin{lemma}\label{E_inverse} The inverse of $E$ is given by
\begin{equation*}
    E^{-1} = \frac{S_{n+d}^{d_1 - 1}}{\det(E)} \frac{\prod_{l=d_1 + 1}^{d} |z_l|^{2(n+d+1)}}{\omega_{n+d}^{d-d_1}} R \end{equation*}
where $R = (R_{i,j})_{i,j = 1}^d$, which for large values of $n$ satisfies the following:
\begin{itemize}
    \item For $1 \leq i \leq d_1$,
    \begin{align*}
        R_{i,i} = \det{(B)} + o(1)
    \end{align*}
    and $R_{i,j} \in o(1)$ when $i \neq j$.
    \item For $d_1 < i \leq d$ and $1 \leq j \leq d_1$,
    \begin{align*}
        R_{i,j} =  \frac{(-1)^{d_1 + i}}{\overline{z}_i^{n+d+1}}  \left[\det{\left(B^{(i)}_{(j)}\right)} + o(1).\right]
    \end{align*} 
    \item Otherwise $d_1 < j \leq d$ and $R_{i,j} \in o\left(\frac{1}{\overline{z}_i^{n+d+1}}\right).$
\end{itemize}

\end{lemma}

\begin{proof}
By Cramer's rule, 
\begin{align*}
    E^{-1} = \frac{1}{\det(E)} \begin{pmatrix}E_{1,1}&E_{1,2}& \dots & E_{1,d} \\ E_{2,1} &E_{2,2}& \dots & E_{2,d} \\
    \vdots & \vdots&\ddots& \vdots \\ 
    E_{d,1} & E_{d,2} & \dots & E_{d,d}
    \end{pmatrix}^T
\end{align*}
with each $E_{i,j}$ the appropriate cofactor.
We first focus on the columns that correspond to the zeros in the unit circle, that is $E_{i,j}$, $1 \leq j \leq d_1$.
Firstly, noticing that $e_{i,j} = k_{n+d}(z_i, z_j)$ and using the definition of a cofactor as a determinant, it follows that the diagonal terms are given by
\begin{align*}
    E_{j,j} = \sum_{\substack{\sigma \in \mathscr{S}\\ \sigma(j) = j}}\left[\sgn (\sigma)\prod_{\substack{l=1 \\ l \neq j}}^{d} k_{n+d}(z_l, z_{\sigma(l)})\right].
\end{align*}

Let us decompose this sum depending on the number of indices a given permutation fixes. As in Lemma \ref{lemma-3.2}, let $\mathscr{A}$ be the set of all permutations such that $\sigma(l) = l$ for every $1 \leq l \leq d_1$, and, for each $0 \leq k < d_1-1$, let $\mathscr{B}_k$ be the set of permutations that fix exactly $k$ of the indices in the set $\{1, \dots, d_1\}\backslash\{j\}$. Then $E_{j,j}$ may be decomposed as the sum of two summands, namely

\begin{equation}\label{e_ii}
\sum_{\substack{\sigma \in \mathscr{A}\\ \sigma(j) = j}}\sgn (\sigma)\prod_{\substack{l=1 \\ l \neq j}}^{d} k_{n+d}(z_l, z_{\sigma(l)}) 
+ \sum_{k=0}^{d_1-2}\sum_{\substack{\sigma \in \mathscr{B}_k\\ \sigma(j) = j}}\sgn (\sigma)\prod_{\substack{l=1 \\ l \neq j}}^{d} k_{n+d}(z_l, z_{\sigma(l)}).
\end{equation}

For the first summand in \eqref{e_ii}, we have that if $l = \sigma(l) < d_1$ we get, by definition,
\begin{align*}
    k_{n+d}(z_l, z_l) = S_{n+d}
\end{align*}
and when $l > d_1$ or $\sigma(l) > d_1$, we get, by Lemma \ref{lemma_k}
\begin{align*}
    k_{n+d}(z_l, z_{\sigma(l)}) = C(n, \sigma, l)\frac{\overline{z}^{n+d+1}_{\sigma(l)} z^{n+d+1}_l}{w_{n+d}},
\end{align*}
with $C(n, \sigma, l)$ such that 
\begin{align*}
    C(n, \sigma, l) \to \frac{1}{\overline{z}_{\sigma(l)} z_l-1} \text{ as } n \to \infty.
\end{align*}
Therefore the first summand of \eqref{e_ii} can be expressed as
\begin{align}\label{e_ii_first_summand}
    S_{n+d}^{d_1 - 1} \frac{\prod_{l=d_1 + 1}^{d} |z_l|^{2(n+d+1)}}{w_{n+d}^{d-d_1}}\left(\sum_{\substack{\sigma \in \mathscr{A}\\ \sigma(j) = j}}\sgn (\sigma)\prod_{\substack{l=d_1 + 1}}^{d} C(n, \sigma, l) \right).
\end{align}

Now, for the second term in \eqref{e_ii}, we notice that it will consist of sums similar to \eqref{e_ii_first_summand} but with some factor of $S_{n+d}$ missing. That factor could be replaced, in some cases, by other ones of the form $k_{n+d}(z_l, z_{\sigma(l)})$ with $ l \neq \sigma(l)$, $l, \sigma(l) \leq d_1$. However, invoking Lemma \ref{k_small}, we know that this will make a small contribution in comparison with $S_{n+d}$.
We note that
\begin{align*}
    C(n, \sigma, l) = \frac{1}{\overline{z}_{\sigma(l)} z_l-1}[1 + r(n,\sigma, l)] \text{ with } r(n, \sigma, l) \to 0 \text{ as } n \to \infty.
\end{align*}
Combining all the results above we get that $ E_{j,j}$ is equal to
\begin{equation*}
S_{n+d}^{d_1 - 1} \frac{\prod_{l=d_1 + 1}^{d} |z_l|^{2(n+d+1)}}{w_{n+d}^{d-d_1}}\left(\sum_{\substack{\sigma \in \mathscr{A}\\ \sigma(j) = j}}\sgn (\sigma)\prod_{\substack{l=d_1 + 1}}^{d} \frac{1}{\overline{z}_{\sigma(l)} z_l-1} + r(n)\right),
\end{equation*}
where $r(n) \to 0$ as $n \to \infty$.
Finally, the left-most term in brackets is just $\det{(B)}$, as a permutation $\sigma \in \mathscr{A}$ is also a bijection in the set $\{d_1 + 1, \dots, d\}\}$, so we have exactly the definition of the determinant.

For the non-diagonal terms, as before, we have

\begin{align*}
        E_{i,j} = \sum_{\substack{\sigma \in \mathscr{S}\\ \sigma(i) = j}}\left[\sgn (\sigma)\prod_{\substack{l=1 \\ l \neq i}}^{d} k_{n+d}(z_l, z_{\sigma(l)})\right] .
\end{align*}

A similar argument reveals a missing factor of $S_{n+d}$ in all terms with respect to $E_{j,j}$, yielding 
\begin{align*}
    E_{i,j} \in o(E_{j,j}).
\end{align*}

Next, we focus on the columns that correspond to the zeros outside the unit  disk, that is, $E_{i,j}$, $d_1 < j \leq d$. We divide these cofactors in two cases, one for zeros on the unit circle, and one for the rest.
In the first case, we have $1 \leq i \leq d_1$, so that, using the definition of the cofactor, we get
\begin{align*}
     E_{i,j} = \sum_{\substack{\sigma \in \mathscr{S}\\ \sigma(i) = j}}\left[\sgn (\sigma)\prod_{\substack{l=1 \\ l \neq i}}^{d} k_{n+d}(z_l, z_{\sigma(l)})\right].
\end{align*}

We further divide this sum in two terms: permutations in the set $\mathscr{A}_i$ and those in $\mathscr{B}_k$. In this case permutations in $\mathscr{A}_i$ fix all indices in $\{1, \dots, d_1\} \backslash \{i\}$ and $\mathscr{B}_{k}$ is defined as before, again avoiding $i$. For the sum of terms in $\mathscr{A}_i$ we get
\begin{align*}
         S_{n+d}^{d_1-1}\sum_{\substack{\sigma \in \mathscr{A}_i\\ \sigma(i) = j}}\left[\sgn (\sigma)k_{n+d}(z_j, z_{\sigma(j)})\prod_{\substack{l=d_1+1\\ l \neq i}}^{d} k_{n+d}(z_l, z_{\sigma(l)})\right].
\end{align*}
Using again Lemma \ref{lemma_k} and a reasoning analogous to the one before we get
\begin{align*}
         S_{n+d}^{d_1-1}\frac{\prod_{l=d_1 + 1}^{d} |z_l|^{2(n+d+1)}}{w_{n+d}^{d-d_1}} \left(\frac{\overline{z}_i}{\overline{z}_j}\right)^{n+d+1} \\ \cdot \sum_{\substack{\sigma \in \mathscr{A}_i\\ \sigma(i) = j}}\left[ \frac{\sgn (\sigma)}{\overline{z}_{\sigma(j)} z_j-1} \prod_{\substack{l=d_1+1 \\ l \neq i}}^{d} \frac{1}{\overline{z}_{\sigma(l)} z_l-1} + r \right],
\end{align*}
where $r=r(n, l, \sigma)$ and we have a factor of $\overline{z}_j^{n+d+1}$ missing with respect to the previous cases as we have eliminated the j-th column in the $E$ matrix, as well as an extra factor $\overline{z}_i^{n+d+1}$, since the i-th column in the $E$ matrix was not eliminated in this case.
For the summands corresponding to $\mathscr{B}_{k}$, one can easily check with a similar procedure that they are smaller than the ones in $\mathscr{A}_i$ by a factor of $S_{n+d}$, therefore we can conclude that $E_{i,j}$ is approximately the product of the two factors, 

\begin{equation*}
 S_{n+d}^{d_1-1}\frac{\prod_{l=d_1 + 1}^{d} |z_l|^{2(n+d+1)}}{w_{n+d}^{d-d_1}}\left(\frac{(-1)^{d_1 + j}}{{\overline{z}_j}^{n+d+1}}\right)\end{equation*} 
 
 and  
 
 \begin{equation*}\sum_{\substack{\sigma \in \mathscr{A}_i\\ \sigma(i) = j}}\left[ (-1)^{d_1 + j } \sgn (\sigma)\frac{\overline{z}_i^{n+d+1}}{\overline{z}_{\sigma(j)} z_j-1} \prod_{\substack{l=d_1+1 \\ l \neq i}}^{d} \frac{1}{\overline{z}_{\sigma(l)} z_l-1} + r(n, \sigma, l)\right].
\end{equation*}

This second factor can be easily checked to be $\det{(B^{(j)}_{(i)})}$, where the factor $(-1)^{d_1 + j}$ is added because we are considering a larger permutation, which changes the sign of the permutation $\sigma$ if we regard it as a bijection on the reduced set.

Finally, for the case $d_1 < i \leq d$, doing a similar expansion of the determinant as before, it is easily seen that for each sum we are missing an exponential factor with respect to the previous terms $E_{i,j}$. Thus in this case
\begin{align*}
    E_{i,j} \in o\left(S_{n+d}^{d_1-1}\frac{\prod_{l=d_1 + 1}^{d} |z_l|^{2(n+d+1)}}{\overline{z}_i^{n+d+1}w_{n+d}^{d-d_1}}\right)
\end{align*}
and the lemma follows transposing the matrix. \end{proof}

We can finally establish Lemma \ref{d_n-lemma}.

\begin{proof}[Proof of Lemma \ref{d_n-lemma}.] By Lemma \ref{projections} we have that
\begin{equation*}
    d_{n+d, n} = \frac{1}{\omega_{n+d}} \sum_{i=1}^d A_{i,n} \overline{z}_i^{n+d},
\end{equation*}
where $A^t_n = E^{-1} \cdot v_0^t$. Now, applying Lemma \ref{E_inverse} it follows that
\begin{equation*}
    A_{i,n} = \frac{S_{n+d}^{d_1 - 1}}{\det(E)} \frac{\prod_{l=d_1 + 1}^{d} |z_l|^{2(n+d+1)}}{\omega_{n+d}^{d-d_1}} \sum_{j=1}^d R_{i,j}.
\end{equation*}
We can plug this value of $A_{i,n}$ into the previous formula, which leads us to \begin{equation*}
    d_{n+d, n} = \frac{S_{n+d}^{d_1 - 1}}{\det(E)} \frac{\prod_{l=d_1 + 1}^{d} |z_l|^{2(n+d+1)}}{\omega_{n+d}^{d - d_1+1}} \sum_{i,j=1}^d R_{i,j} \overline{z}_i^{n+d}.
\end{equation*}
We start working with the multiplicative term. Lemma \ref{lemma-3.2} actually tells us that
\begin{equation*}
    \det{(E)} = C S_{n+d}^{d_1}\frac{\prod_{l=d_1 + 1}^{d} |z_l|^{2(n+d+1)}}{\omega_{n+d}^{d-d_1}} (1 + o(1))
\end{equation*}
for some positive constant $C$. We can substitute this value into the formula for $d_{n+d,n}$ to obtain
\begin{equation*}
    d_{n+d, n} = \frac{C}{S_{n+d} \omega_{n+d}}(1 + o(1))\sum_{i,j=1}^d R_{i,j} \overline{z}_i^{n+d}.
\end{equation*}
We can divide the sum $\sum_{i,j=1}^d R_{ij} \overline{z}_i^{n+d}$ in 3 different terms, as
\begin{equation*}
\sum_{i=1}^{d_1} \sum_{j=1}^d R_{i,j} \overline{z}_i^{n+d} + \sum_{i=d_1+1}^{d} \sum_{j=1}^{d_1} R_{i,j} \overline{z}_i^{n+d} + \sum_{i=d_1+1}^{d} \sum_{j=d_1+1}^{d} R_{i,j} \overline{z}_i^{n+d}.
\end{equation*}
For the first of these 3 sums, we can directly apply Lemma \ref{E_inverse} to get
\begin{align*}
    \sum_{i=1}^{d_1} \sum_{j=1}^d R_{i,j} \overline{z}_i^{n+d} = \sum_{i=1}^{d_1} R_{i,i} \overline{z}_i^{n+d} + \sum_{i=1}^{d_1} \sum_{j \not = i}^d R_{i,j} \overline{z}_i^{n+d} \\
    = \left(\sum_{i=1}^{d_1}\overline{z_i}^{n+d}\right) \det{(B)} + o(1).
\end{align*}

Using again Lemma \ref{E_inverse}, we can see that the second summand is equal to
\begin{align*} \sum_{i=1}^{d-d_1} \frac{(-1)^{i}}{\overline{z}_{i+d_1}} \left(\sum_{j=1}^{d_1} \det{\left(B^{(i + d_1)}_{(j)}\right)}\right) + o(1).
\end{align*}

The last summand is simply $o(1)$ since the values $R_{ij}$ are of order $o(1/\overline{z}^{n+d+1}_{i})$. The result follows.\end{proof}

We are also ready for a proof of Lemma \ref{g_one_zero}.

\begin{proof} [Proof of Lemma \ref{g_one_zero}.]
Let us first express every quantity in terms of $a_j = \frac{1}{\overline{z}_j} \in \mathbb{D}$, $j = 2, \dots, d$ so that $v_m = a_m$, $b_{l,m} = \frac{\overline{a_l} a_m}{1 - \overline{a}_{l} a_{m}}$ and $s_l = \frac{\overline{a}_l}{1 - \overline{a}_l}$.
From the expansion of the determinant along the $l$-th row, we can remove a factor of $\overline{a}_l$ and the same holds for $a_m$ along the $m$-th column, yielding
\begin{equation*}
    G = \left(\prod_{j=2}^d |a_j|^2 \right) \begin{vmatrix} 1 & v_0\\ p^t & H \end{vmatrix},
\end{equation*}
where $v_0 = (1, \dots, 1) \in \mathbb{C}^{d-1}$, $p = \left( p_l \right)_{l=2}^d$, $p_l = \frac{1}{1 - \overline{a}_l}$ and $H = \left(h_{l,m}\right)_{l,m=2}^{d}$, $h_{l,m} = \frac{1}{1 - \overline{a}_l a_m}$. 
The first product is never zero, as $a_j \not = 0$, $j=2, \dots, d$, so we only have to prove that the determinant on the right side is non-zero.

The determinant is the value at $z=1$ of the function $g$ given by the orthogonal projection in $H^2$ of the constant function $1$ onto $\spn\{k_{a_j}(z)\}_{j=2}^d =: V$, where $k_w(z) = \frac{1}{1 - \overline{w}z}$ is the Szeg\"o kernel at $w$ evaluated at $z$ ($w,z \in \mathbb{D}$). 

From this property of projections, it is standard that 
\begin{equation*}
    g(1) = 1 - \prod_{j=2}^d \left(-\overline{a}_j\right) \left(\frac{1 - a_j}{1 - \overline{a}_j}\right).
\end{equation*}
Now notice that $\left| \frac{1 - a_j}{1 - \overline{a}_j} \right | = 1$, while $|a_j| < 1$, so the product in the right hand side is strictly smaller than 1. In other words $g(1) \not = 0$.
\end{proof}

The proposition we need about angles is elementary, being just a particular case of Kronecker's approximation theorem, but we decide to provide a proof for completeness:

\begin{proof}[Proof of Lemma \ref{sum-cosines-lemma}.] Without loss of generality we may assume that all $\theta_i$'s are different. Moreover, from now on, when talking about angles we will restrict ourselves to the interval $[-\pi, \pi)$ with the corresponding mod $2\pi$ association.
Given $\epsilon > 0$ it is enough to find a number $K \in \mathbb{N}$ such that the angles $K\theta_1$, $K\theta_2$ are in $[-\epsilon, \epsilon]$.

Fix $T \in \mathbb{N}$ such that $1/T \leq \epsilon$, and choose $k_1$ such that $\theta'_1 \equiv \theta_1 k_1$,  satisfies $|\theta_1'| \in \left[\frac{-1}{2T^3}, \frac{1}{2T^3}\right]$. 
Then $\{s\theta_{1}'\}_{s=1}^{2T^2}$ is a sequence of $2T^2$ points all of which are in $[-\epsilon, \epsilon]$. Consider the sequence $\{k_1s\theta_2\}_{s=1}^{2T^2}$ which is a set with $2T^2$ elements. From the pigeonhole principle, it follows that there exist $s_1$ and $s_2$ such that $|(s_1 - s_2)k_1\theta_2| \leq \frac{1}{T} = \epsilon$.  Take $k_2 = |s_1 - s_2|$ and $K = k_1k_2$ the results follows. 

The general case follows automatically via induction. \end{proof}

Lastly, we must understand linear combinations of unimodular complex numbers:

\begin{proof}[Proof of Lemma \ref{sum_exponentials_lemma}.]
Suppose we find some $C_i \not = 0$ such that   \begin{equation*}
        \sum_{k=1}^n C_k e^{im\theta_k} = 0, \forall m \geq N.
    \end{equation*}
    Without loss of generality, we can assume that $C_1 \not = 0$ and rearrange the equation to get
    \begin{equation*}
        \sum_{k=2}^n (-C_k/C_1) e^{im(\theta_k - \theta_1)} = 1, \forall m \geq N.
    \end{equation*}
    As this is true for all $m \geq N$, we can sum the different equalities in $m$ to get
    \begin{align*}
        R &= \sum_{m=N}^{N+R} 1 = \sum_{m=N}^{N+R} \sum_{k=2}^n (-C_k/C_1) e^{im(\theta_k - \theta_1)} \\
        &= \sum_{k=2}^n (-C_k/C_1) \sum_{m=N}^{N+R} e^{im(\theta_k - \theta_1)} \leq \sum_{k=2}^n |C_k/C_1| \frac{2}{|1 - e^{i(\theta_k - \theta_1)}|}
    \end{align*}
    Taking the limit when $R \to \infty$ yields the desired contradiction, as the right side is just a constant.
\end{proof}

\section{Further remarks}\label{sect5}

From where we stand it is not difficult to draw a few further directions of research.

One may consider whether our results hold true as well for the zeros of o.p.a. $\{p_n\}_{n\in \N}$ rather than those of $1-p_nf$. This is the equidistribution that is suggested in the original paper \cite{BCLSS} and perhaps a result of interlacing would give this additional asymptotic result.

In some of our theorems, the appearance of subsequences is somewhat surprising. One may wonder whether the propositions do hold without needing to take subsequences, but it seems from the calculations that there are some obstructions. Perhaps this is an actual opportunity: it seems plausible that one could construct some special functions $f$ as limits of other polynomials $f_n$, for which the corresponding subsequences have larger and larger gaps, making $f$ bear some special property with regards to cyclicity, precisely because no subsequence of the natural numbers would work for all $f_n$.

Another natural research line could focus on the strong asymptotics as suggested by Rakhmanov, or on the value distribution or level sets of $1-p_nf$ over $\T$, in general. At least in the Dirichlet space, this can be linked to capacities and the space norm of $1-p_nf$ in a way that connects directly with the study of cyclic functions.

Finally, it seems desirable to remove our requirement that the zeros of the function in study be simple. Although we believe this to be possible, we decided not to include such complications in here. The recent article by Felder and Le \cite{FelderLe} shows how to deal with multiple zeros in a very similar context.

\bigskip

\noindent\textbf{Acknowledgements.} We were supported by the ``Severo Ochoa Programme for Centers of Excellence in R\&D'' (CEX2019-000904-S); by a grant from Agencia Estatal de Investigaci\'on (PID2019-106433GB-I00/AEI/10.13039/501100011033), Spain; and by the Madrid Government (Comunidad de Madrid-Spain) under the Multiannual Agreement with UC3M in the line of Excellence of University Professors (EPUC3M23), and in the context of the V PRICIT (Regional Programme of Research and Technological Innovation).

\end{document}